\documentclass[10pt]{amsart}
\usepackage{amssymb}
\usepackage{enumerate}
\usepackage{color}

\usepackage{graphicx}
\usepackage{xcolor}

\newcommand{\res}{\upharpoonright}

\newcommand{\Cbb}{\mathbb{C}}
\newcommand{\C}{\Cbb}
\newcommand\Nbb{\mathbb{N}}
\newcommand\Sbb{\mathbb{S}}
\newcommand\Tbb{\mathbb{T}}
\newcommand\Zbb{\mathbb{Z}}

\newcommand\aut{\operatorname{Aut}}
\newcommand{\Seq}[1]{\langle #1 \rangle}

\newcommand\supt[1]{[ #1 ]}

\theoremstyle{plain}
\newtheorem{theorem}{Theorem}

\newtheorem{lemma}{Lemma}

\theoremstyle{remark}

\newtheorem{claim}{Claim}

\newcommand{\od}{\overline{\bigodot}}
\newcommand{\cl}{\operatorname{cl}}

\begin{document}

\title[Spectral form of Koopman representations]{The spectral form of Koopman representations of 
the group of measurable functions\\ with values in the circle}

\author{Justin Tatch Moore \and S{\l}awomir Solecki}

\thanks{Moore's research supported by NSF grant DMS-1854367.}

\thanks{Solecki's research supported by NSF grant DMS-1954069.}


\address{Department of Mathematics\\
Cornell University\\
Ithaca, NY 14853}

\email{justin@math.cornell.edu}

\email{ssolecki@cornell.edu}



\begin{abstract}
We compute the spectral form of the Koopman representation induced by a natural boolean action of $L^0(\lambda, \Tbb)$ identified earlier by the authors.
Our computation establishes the sharpness
of the constraints on spectral forms of Koopman representations of $L^0(\lambda, \Tbb)$ 
previously found by the second author. 
\end{abstract}

\maketitle

\section{Introduction} 

Broadly speaking, this paper is about unitary representations of a specific non-locally compact topological group:
$L^0(\lambda, \Tbb)$.
This group consists of all $\lambda$-measurable functions 
from $2^\Nbb$ to $\Tbb$
equipped with the operation of point-wise multiplication and the topology of convergence in measure.
Here $\lambda$ is the countable power of the uniform measure on $\{0,1\}$
and $\Tbb := \{z \in \Cbb \mid |z| = 1\}$. 

The group $L^0(\lambda, \Tbb)$ 
is an important example in topological dynamics --- it is both an extremely amenable group 
\cite{Gla} and the simplest example of a Polish group whose boolean ergodic actions are not point realized \cite{GW}. 
The underlying cause of both phenomena is 
it being a L{\'e}vy group.
It occurs naturally as the closed group generated by 
a generic unitary operator \cite{MT} and as the closed group generated by Gaussian transformations \cite{LPT}.

The main theorem of \cite{Sol} shows that 
all unitary representations of $L^0(\lambda, \Tbb)$ have a specific form, which we 
call the {\em spectral form of the representation}; see Section \ref{Su:spfr} below.
Then in \cite{Sol2} a new constraint on the 
spectral forms of {\em Koopman} representations of $L^0(\lambda, \Tbb)$ was established;
see Section \ref{Su:sem} below.
This is in contrast to the spectral forms of {\em arbitrary}
unitary representations of $L^0(\lambda, \Tbb)$.
This additional constraint was the key ingredient in proving  
that the closed group generated by a generic measure preserving transformation 
is not isomorphic to $L^0(\lambda, \Tbb)$ \cite{Sol2}.
Up to this point, however,
there has been no nontrivial example of a concrete unitary representation of $L^0(\lambda, \Tbb)$
for which the spectral form 
has been computed; in fact, the spectral form has not be computed for any Koopman representation associated 
with a boolean ergodic action of $L^0(\lambda, {\mathbb T})$.

In the present paper, we consider a natural boolean ergodic action 
of $L^0(\lambda, \Tbb)$ defined in \cite{MoSo}
in connection with the problem of finding point realizations of boolean actions of the
subgroup $C(2^{\Nbb},\Tbb)$. 
We recall the definition of this action in Section~\ref{Su:act}. 
Theorem~\ref{T:koo} 
of Section~\ref{Su:act2} explicitly computes the spectral form of the Koopman representation induced by this action.
This is the main result of the current paper.
Our Theorem~\ref{T:koo} shows that the constraint found in \cite{Sol2} is essentially strongest possible.
We discuss these issues in Section~\ref{Su:sem}.
Section~\ref{S:prtr} is devoted to proving Theorem~\ref{T:koo}.

\section{Background and the statement of the main theorem}

We fix a few notational conventions.
We will use $\Nbb$ to denote the positive integers.
All counting and indexing will start at $1$ except when otherwise specified.
The letters $i,j,k,l,m,n,p,q$ will be used to denote elements of $\Zbb$.
If $n \geq 0$, we will write $2^n$ to denote the collection of all binary sequences of length $n$.
If $s \in 2^n$, $s_i$ will denote the $i$th member of the sequence $s$, recalling our convention that indexing starts with $1$.
If $s,t \in 2^n$, $st$ will denote concatenation of $s$ and $t$.
The collection of all finite binary sequences, including the null sequence, will be denoted by $\Sigma$.
We will write $\Zbb^\times$ to denote the set of all nonzero integers, regarded as a semigroup with the operation
of multiplication.
Furthermore, $\Nbb[\Zbb^\times]$ will denote the collection of all functions whose domain is a non-empty
finite subset of $\Zbb^\times$ and whose values are in $\Nbb$.
Unless stated otherwise, $x$ and $y$ will denote elements of $\Nbb[\Zbb^\times]$.

\subsection{The spectral form of unitary representations of $L^0(\lambda, \Tbb)$}\label{Su:spfr} 

The following result from \cite{Sol} describes the form of an arbitrary unitary representation of $L^0(\lambda, \Tbb)$. 
We state the result first and then define the notation and terminology 
needed to understand its statement.

\begin{theorem}[\cite{Sol}]\label{T:rep} 
Let $\xi\colon L^0(\lambda, \Tbb)\to {\mathcal U}(H)$ be a unitary representation on a separable Hilbert space $H$. 
Let $H_0$ be the orthogonal complement of the subspace of $H$ consisting of vectors fixed by the representation. 
Then $\xi$ restricted to $H_0$ is isomorphic to the $\ell^2$-sum over $x\in \Nbb [\Zbb^\times]$ and 
$j\in \Nbb$ of the representations $\rho_{\mu^j_x}$ associated to 
a sequence of finite measures  $\mu^j_x$ indexed by $x\in \Nbb[\Zbb^\times]$ and $j \in \Nbb$
such that, for each $x$ and $j$:
\begin{itemize}

\item
$\mu_x^j$ is compatible with $x$;

\item $\mu^{j+1}_x\preceq \mu^j_x$. 

\end{itemize}
The measures $\mu^j_x$ are uniquely determined up to mutual absolute continuity. 
\end{theorem}

The $\ell^2$-sum satisfying the conclusion of Theorem~\ref{T:rep} is called the {\bf spectral form of the representation} $\xi$.
The spectral form determines the representation $\xi$ on $H_0$ up to isomorphism.
Moreover, the spectral form is determined by the sequence of measures 
$(\mu^j_x)_{x,j}$,
which we call the {\bf sequence of measures determining the spectral form}.

We now turn to the task of setting up the terminology and notation used in
Theorem \ref{T:rep}. 
For $x\in \Nbb[\Zbb^\times]$, we write 
\begin{align*}
D(x) := & \{ (k, i)\mid k\in {\rm dom}(x),\, 0\leq i<x(k)\} \\
C_x := & (2^\Nbb)^{D(x)}.
\end{align*}
Define $\pi_{k, i}\colon C_x \to 2^\Nbb$, for $(k, i)\in D(x)$, to be the projection 
from $C_x$ onto coordinate $(k,i)$.

For $x\in \Nbb[\Zbb^\times]$, 
a permutation $\delta$ of $D(x)$ is called {\bf good} if, for each $(k,i)\in D(x)$, 
$\delta(k,i)= (k,j)$ for some $j$. Each good permutation $\delta$ induces a homeomorphism $\widetilde \delta$ of 
$C_x= (2^\Nbb)^{D(x)}$ 
by permuting coordinates 
\begin{equation}\notag
\pi_{\delta(k,i)} \big( {\widetilde \delta}(\alpha)\big) = \pi_{k,i}(\alpha), \hbox{ for all } \alpha \in C_x \hbox{ and } (k,i)\in D(x).
\end{equation} 
We call such homeomorphisms $\widetilde \delta$ {\bf good homeomorphisms of $C_x$}.

Let $x \in \Nbb[\Zbb^\times]$ and let $\mu$ be a finite Borel measure on $C_x$. We say that $\mu$ is {\bf compatible with} $x$ if 
\begin{enumerate}[(i)]

\item \label{marginal}
the marginal measures $(\pi_{k,i})_*(\mu)$ of $\mu$ on $2^\Nbb$
are absolutely continuous with respect to $\lambda$ for each $(k,i) \in D(x)$; 

\item \label{good_inv}
$\mu$ is invariant under good homeomorphisms of $C_x$; 

\item \label{off_diag}
all sets of the form $\{ \alpha\in C_x\mid \pi_{k,i}(\alpha) = \pi_{k',i'}(\alpha)\}$, 
for distinct $(k,i), (k',i')\in D(x)$, 
have measure zero with respect to $\mu$. 

\end{enumerate}
Condition (\ref{marginal}) is needed for representations as in \eqref{E:pire} below to be well defined, 
while conditions (\ref{good_inv}) and (\ref{off_diag}) ensure uniqueness in Theorem~\ref{T:rep}.

Let $x\in \Nbb[\Zbb^\times]$ and fix a measure $\mu_x$ which is compatible
with $x$.
Define a homomorphism $R_x : L^0(\lambda,\Tbb) \to L^0(\mu_x,\Tbb)$ by
\begin{equation}\notag
R_x(\phi) := \prod_{(k,i)\in D(x)} (\phi\circ \pi_{k, i})^k. 
\end{equation}
Since $\phi$ is a measure class of functions with respect to $\lambda$ and $\mu_x$ is compatible with $x$,
$R_x(\phi)$ determines a measure class of functions 
with respect to $\mu_x$.

By \cite{Sol}, for all $\phi \in L^0(\lambda,\Tbb)$ the function $R_x(\phi)$
is invariant under all good homeomorphisms of $C_x$. 
Define $\widetilde{L^2}(\mu_x)$
to be the closed subspace of $L^2(\mu_x)$ consisting of all (equivalence classes of) 
functions invariant under good homeomorphisms of $C_x$. 
Since $R_x(\phi)$ is invariant under all good homeomorphisms of $C_x$, we have that
for all $f\in \widetilde{L^2}(\mu_x)$, 
the product $R_x(\phi) f$  is an element of $\widetilde{L^2}(\mu_x)$. 
Let 
$\rho_{\mu_x}(\phi)$ be the multiplication operator on $\widetilde{L^2}(\mu_x)$ given by
\begin{equation}\label{E:pire}
\rho_{\mu_x}(\phi)\big( f\big) = R_x(\phi) f,\; \hbox{ for }f \in \widetilde{L^2}(\mu_x).
\end{equation} 
It is easy to check that $\rho_{\mu_x}(\phi)$ 
is a unitary operator on $\widetilde{L^2}(\mu_x)$ and that 
\begin{equation}\notag
\rho_{\mu_x} \colon L^0(\lambda, \Tbb)\to {\mathcal U}\big( \widetilde{L^2}(\mu_x)\big)
\end{equation}
is a unitary representation of $L^0(\mu_x, \Tbb)$.

\subsection{The boolean action of $L^0(\lambda, \Tbb)$}\label{Su:act} 

We will now recall some notation and definitions from \cite{MoSo}.
We will use $\gamma$ to denote the probability measure on $\C$ such that the real and imaginary components of 
$\gamma$ are independent Gaussian random variables with mean $0$ and variance $1/2$---giving a probability density 
function of $(1/\pi) e^{-x^2 - y^2}$. Observe that $\gamma$ is invariant under the action of $\Tbb$ on $\C$.

Recall that if $k \geq 0$, $2^k$ denotes the set of all binary sequences of length $k$
and $\Sigma$ denotes the set of all finite binary sequences.
Define $X_k :=\C^{2^k}$ equipped with the product measure $\gamma_k$ obtained
by equipping each coordinate with $\gamma$.
Set
\[
X_\infty :=
\{f \in \C^{\Sigma} \mid \forall s \in \Sigma \ \ f(s) = \frac{f(s 0) + f(s 1)}{\sqrt{2}} \} 
\]
and for each $s \in \Sigma$, let $z_s\colon X_\infty\to \C$ denote the evaluation function $f \mapsto f(s)$.
The measure $\gamma_\infty$ on $X_\infty$ is the unique measure which, for each $k$, pushes forward to
 $\gamma_k$ via the restriction map $f \mapsto f \restriction 2^k$.
  
Observe that for each $k \geq 0$,
$\{z_s \mid s \in 2^k\}$ is both an orthonormal family in $L^2(\gamma_k)$ and a mutually independent
family of identically distributed random variables.
We also note that for any $n \geq 0$, $s,t \in 2^n$ and $k,m \geq 0$:
\begin{equation} \label{inner_prod}
\int_{X_\infty} z^k_s \bar z^m_t \ d \gamma_\infty = 
\begin{cases}
m! & \textrm{ if } s = t \textrm{ and } k = m \\
1 & \textrm{ if }  k=m = 0 \\
0  & \textrm{ otherwise}
\end{cases}
\end{equation}
When $k=m$ and $s=t$, this reduces to the calculation
\begin{align*}
\frac{1}{\pi} \int_{-\infty}^\infty\!\int_{-\infty}^\infty  (x^2 + y^2 )^m e^{- x^2 - y^2}\ dx\ dy & = 
\frac{1}{\pi} \int_0^{2\pi} d\theta \int_{0}^\infty r^{2m+1} e^{-r^2}\ dr \\
& = \int_0^\infty u^m e^{-u}\ du = m!.
\end{align*}
If $s=t$ but $k \ne m$, then the integral is a multiple of $\int_0^{2\pi} e^{(k-m)i \theta} \ d\theta = 0$ and if
$s \ne t$, the independence of the random variables $z_s$ and $z_t$ implies 
$$\int_{X_\infty} z^k_s \bar z^m_t \ d \gamma_\infty =
\int_{X_\infty} z^k_s \ d\gamma_\infty \cdot \int_{X_\infty} \bar z^m_s \ d\gamma_\infty$$
which can then be evaluated as in the previous cases with one or both exponents replaced by $0$.

Next we will recall the boolean action of $L^0(\lambda,\Tbb)$ on $(X_\infty,\gamma_\infty)$.
In order to do this, we need to make some preliminary definitions.
For a finite binary sequence $s\in 2^n$, set 
\begin{equation}\notag 
[s] := \{ a \in 2^\Nbb\mid a \res n = s\}. 
\end{equation} 
Given $n$, we write $\Sbb_n$ for the subgroup of $L^0(\lambda, \Tbb)$ consisting of all functions 
constant on each of the sets $[s]$ for $s\in 2^n$.
As a topological group $\Sbb_n$ is isomorphic to $\Tbb^{2^n}$. 
The obvious canonical isomorphism between these two groups allows us to 
view $g$ as a sequence $(g_s)_{s\in 2^n}$ with $g_s\in \Tbb$ and, at the same time, as 
a function on $2^\Nbb$ defined by the formula $g(x)=g_s$ for $x\in [s]$. 
Observe that 
\[
\Sbb := \bigcup_n \Sbb_n. 
\]
is a dense subgroup of $L^0(\lambda, \Tbb)$. 

The action of $\Sbb$ on $X_\infty$ is as follows.
If $g \in \Sbb_n$ and $f \in X_\infty$
\[
(g \cdot f)(s) :=
\begin{cases}
g_s \cdot f(s) & \textrm{ if } |s| \geq n \\
\frac{1}{\sqrt{2}^k} {\displaystyle \sum_{t \in 2^k}} (g \cdot f)(st) & \textrm{ if } |s|=n-k < n
\end{cases}
\] 
This formula defines an action of $\Sbb$ on the algebra of Borel subsets of $X_\infty$
modulo the $\gamma_\infty$-measure zero sets.
Such a {\bf boolean action} can be viewed as a continuous embedding of $\Sbb$ into
the group $\aut(\gamma_\infty)$ of measure preserving transformations of $\gamma_\infty$.
The embedding extends uniquely to a continuous embedding of $L^0(\lambda,\Tbb)$ into
$\aut(\gamma_\infty)$.
This is the boolean action of interest to us.

\subsection{The spectral form of the Koopman representation induced by the boolean action}\label{Su:act2} 

The following theorem is the main result of the paper. 

\begin{theorem}\label{T:koo}
The spectral form of the Koopman representation induced by the boolean action of 
$L^0(\lambda, \Tbb)$ 
on $(X_\infty, \gamma_\infty)$ is determined by the sequence of measures 
$(\mu^j_x)_{x,j}$
described as follows: 
\begin{itemize}

\item if ${\rm dom}(x)\subseteq \{ -1, 1\}$, then $\mu^1_x=\lambda^{D(x)}$;

\item if ${\rm dom}(x) \not \subseteq \{ -1,1\}$ or $j \ne 1$, then $\mu^j_x = 0$.

\end{itemize}
\end{theorem}
We will now rephrase Theorem~\ref{T:koo} in even more explicit terms. 
For $p,q\in \Zbb$ with $p,q\geq 0$ and $p+q>0$, define $x_{p,q}\in \Nbb[\Zbb^\times]$ 
by 
\[
x_{p,q}(k) := 
\begin{cases}
p & \textrm{ if } k = 1 \\
q & \textrm{ if } k =-1 \\
0 & \textrm{ otherwise}
\end{cases}
\]
With this notation Theorem~\ref{T:koo} says that, in the sequence $(\mu^j_x)_{x,j}$
in the spectral form of the Koopman representation of 
the boolean action of $L^0(\lambda, \Tbb)$ on $(X_\infty, \gamma_\infty)$, 
the only non-zero measures are $\mu^1_{x_{p,q}}$ and they are equal to $\lambda^{p+q}$.

The measure $\mu^1_{x_{p,q}}$ is defined on the space $C_{x_{p,q}}$; 
the space can be identified with $(2^\Nbb)^p\times (2^\Nbb)^q$ and the measure with $\lambda^{p+q}$. 
We introduce the following notation 
\[
\mu_{p,q}:= \mu^1_{x_{p,q}} \qquad \qquad C_{p,q}:= C_{x_{p,q}} \qquad \qquad \rho_{p,q}:=\rho_{\mu_{p,q}}.
\]
The Hilbert space $\widetilde{L^2}(\mu_{p,q})$ consists of all functions in $L^2(\mu_{p,q})$ invariant under the permutation 
of the first $p$ coordinates of $C_{p,q}$ and the last $q$ coordinates of $C_{p,q}$. 
The representation 
$\rho_{p,q} : L^0(\lambda,\Tbb) \to U\big(\widetilde{L^2}(\mu_{p,q})\big)$ is given by the formula 
\begin{equation}\label{rho_eq}
\rho_{p,q}(g)(f)(a,b) = \big(g(a_1)\cdots g(a_p)\cdot g(b_1)^{-1} \cdots g(b_q)^{-1}\big) f(a,b), 
\end{equation} 
where $g \in L^0(\lambda, \Tbb)$, $a= (a_1, \dots , a_p)$ and $b= (b_1, \dots , b_q)$ with $(a,b)\in C_{p,q}$.
(Here $\rho_{p,q}(g)$ is a unitary operator whose
value $\rho_{p,q}(g)(f)$ at $f \in \widetilde{L^2}(\mu_{p,q})$ is specified pointwise via (\ref{rho_eq}).)

Finally, we  state Theorem~\ref{T:koo} explicitly using the notation above. This is the statement that we will actually prove. Proving 
it suffices to justify Theorem~\ref{T:koo} by the uniqueness clause of Theorem~\ref{T:rep}. 

\begin{theorem}\label{T:exp} 
The Koopman representation of 
the boolean action of $L^0(\lambda, \Tbb)$ on $(X_\infty, \gamma_\infty)$ 
restricted to the orthogonal complement of constant functions is isomorphic to the $\ell^2$-sum of the representations 
$\rho_{p,q}$ over all $p,q\in \Zbb$, with $p,q \geq 0$ and $p+q>0$.
\end{theorem}

\section{The proof of the main theorem}\label{S:prtr} 

Let $A$ be the $\ell^2$-sum of the Hilbert spaces $\widetilde{L^2}(\mu_{p,q})$ over $p,q\in \Zbb$ with 
$p,q\geq 0$ and $p+q>0$. Let 
\[
\alpha \colon L^0(\lambda, \Tbb)\to {\mathcal U}(A)
\]
be the unitary representation arising from the representations $\rho_{p,q}$ in (\ref{rho_eq}).
Let $B$ be the Hilbert space that is the orthogonal complement of constant functions in $L^2(\gamma_\infty)$. Let 
\[
\beta \colon L^0(\lambda, \Tbb)\to {\mathcal U}(B)
\]
be the Koopman representation arising from the boolean action on $(X_\infty, \gamma_\infty)$.

We will prove Theorem~\ref{T:exp} by showing that there exists a Hilbert space isomorphism $\Phi\colon A\to B$ which
is equivariant between the restrictions of $\alpha$ and $\beta$ to $\Sbb$;
this is sufficient since $\Sbb$ is dense in $L^0(\lambda,\Tbb)$.
Rather than construct $\Phi$ directly, we will proceed as follows. 
We will build a sequence of Hilbert spaces $\Gamma(n)$ and embeddings
$E_n:\Gamma(n) \to \Gamma(n+1)$, defined for $n \in \Nbb$. The spaces $\Gamma(n)$ are variations of Fock spaces. 
For each $n \in \Nbb$, we will also construct a representation
$\psi_n\colon \Sbb_n\to {\mathcal U}(\Gamma(n))$ so that 
\[
E_n\circ \psi_n = \psi_{n+1}\res \Sbb_n
\]
as well as functions
\[
F_n^\alpha \colon \Gamma(n) \to A\;\hbox{ and }\; F_n^\beta \colon \Gamma(n)\to B
\]
with the following properties: 
\begin{enumerate}[(a)]

\item \label{F_equi}
$F_n^\alpha$ and $F_n^\beta$ are $\Sbb_n$-equivariant between $\psi_n$ and $\alpha\res \Sbb_n$ and 
$\psi_n$ and $\beta \res \Sbb_n$, respectively; 

\item \label{F_cohere}
$F_n^\alpha = F_{n+1}^\alpha \circ E_n$ and $F_n^\beta = F_{n+1}^\beta \circ E_n$; 

\item \label{F_dense}
$\bigcup_n F_n^\alpha \big(\Gamma(n)\big)$ and $\bigcup_n F_n^\beta \big(\Gamma(n)\big)$ are dense in $A$ and $B$, respectively. 
\end{enumerate}

Strictly speaking, involving the $\Gamma(n)$ is not necessary: it would not be difficult to define $\Phi$ directly and
eliminate the reference to them.
Interpolating through the $\Gamma(n)$, however, does help abstract and separate the tasks in the proof; 
this is the reason we have chosen the current approach.

In order to see that this is sufficient to prove Theorem~\ref{T:exp},
set 
$A(n) := F_n^\alpha (\Gamma(n))$ and
$B(n) := F_n^\beta (\Gamma(n))$.
Clearly $A(n)$ and $B(n)$ are subspaces of $A$ and $B$, respectively, and by (\ref{F_cohere})
\[
A(n)< A(n+1)\;\hbox{ and }\; B(n)<B(n+1). 
\]
Define $\Phi_n\colon A(n)\to B(n)$ by 
\[
\Phi_n = F_n^\beta\circ (F_n^\alpha)^{-1}. 
\]
Then $\Phi_n$ is a Hilbert space isomorphism. By (\ref{F_cohere}), $\Phi_{n+1}$ extends $\Phi_n$. Thus, 
by (\ref{F_dense}), there exists a Hilbert space isomorphism $\Phi\colon A\to B$ that extends each $\Phi_n$. 
Further, by (\ref{F_equi}), $\Phi_n$ is $\Sbb_n$-equivariant between $\alpha \res \Sbb_n$ and 
$\beta \res \Sbb_n$. Therefore, $\Phi$ is $\Sbb$-equivariant between $\alpha$ and $\beta$, as required. 

Therefore, to prove Theorem~\ref{T:exp}, it remains to complete the following steps: 
\begin{enumerate}[1.]
\item \label{GammaEpsi}
Define $\Gamma(n)$, $E_n$, and $\psi_n$ and prove their relevant properties.

\item \label{F^a_n}
Define $F_n^\alpha$ and prove their relevant properties.

\item \label{F^b_n}
Define $F_n^\beta$ and prove their relevant properties.
\end{enumerate}

\subsection*{Step \ref{GammaEpsi}: $\Gamma(n)$, $E_n$, and $\psi_n$ and their properties}\label{S:foc} 

We will now recall the construction of symmetric Fock space and describe a modification
which we will need.
For a non-empty finite set $S$ with the normalized counting measure we consider $L^2(S)$, the space of all complex valued functions on $S$ with the $L^2$-norm.
For $s \in S$,  set $v_s := \sqrt{|S|}\chi_{\{ s\}}$. 
The set $\{v_s \mid s\in S\}$ is an orthonormal basis of $L^2(S)$. 

We consider the symmetric tensor product $L^2(S)^{\bigodot l}$.
This is done as in \cite[Appendix E]{Jan}.
One first forms the tensor product $L^2(S)^{\bigotimes l}$ 
which is the Hilbert space with orthonormal basis consisting of all vectors of the form
$v_{s_1}\otimes \cdots \otimes v_{s_l}$ for $s_1, \dots , s_l\in S$.
Now $L^2(S)^{\bigodot l}$ is the subspace of $L^2(S)^{\bigotimes l}$ spanned by vectors of the form 
\[
v_{s_1}\cdots v_{s_l} = \frac{1}{\sqrt{l!}} \sum_{f\in {\rm Sym}(l)} v_{s_{f(1)}}\otimes \cdots \otimes v_{s_{f(l)}}.
\]
One checks that the vector $v_{s_1}\cdots v_{s_l}$ does not depend on the order of $s_1, \dots, s_l$ and that 
vectors of this form are orthogonal to each other. 
Thus $L^2(S)^{\bigodot l}$ has an orthogonal basis consists of all vectors of the form $v_{s_1}\cdots v_{s_l}$ for 
$s_1\leq \cdots \leq s_l$, where $\leq$ is a fixed linear order on $S$.
We will refer to such vectors as {\bf basic product vectors}.
It may be verified that
\begin{equation}\label{E:nor}
\| v_{s_1}\cdots v_{s_l}\| = (m_1!\, \cdots \, m_{l'}!)^{1/2}, 
\end{equation}
where $l'$ is the number of distinct entries in the sequence $s_1, \dots, s_l$, say these are $t_1, \dots, t_{l'}$, and 
\[
m_i = |\{ j\leq l\colon t_i=s_j\}| ,\hbox{ for }1\leq i\leq l'. 
\]

Now, we consider a non-empty finite set $S$ together with a disjoint copy $\overline{S}$ of $S$.
We also fix an involution $\sigma \mapsto \overline{\sigma}$ of
$S \cup \overline{S}$ which maps $S$ to $\overline{S}$ and
vice-versa. 
We say that a sequence $\sigma_1,\dots,\sigma_l$ is {\bf admissible} if for all $1 \leq i,j \leq l$,
$\sigma_i \ne \overline{\sigma_j}$.
We define the subspace
\[
L^2(S\cup \overline{S})^{\od l}
\]
of the symmetric tensor product $L^2(S\cup \overline{S})^{\bigodot l}$ to be the span of all
basic product vectors $v_{\sigma_1} \cdots v_{\sigma_l}$ where $\sigma_1,\dots,\sigma_l$ is admissible.
Following \cite[Appendix E]{Jan}, we define the modified symmetric Fock space 
$\Gamma(S, \overline{S})$ to be the the $\ell^2$-direct sum of the spaces
$L^2(S\cup\overline{S})^{\od l}$ as $l$ ranges over the positive integers.

\subsection*{Definition of $\Gamma(n)$} 
Fix a disjoint copy $\overline{2^n}$ of $2^n$ and an involution $\sigma \mapsto \overline{\sigma}$ as above and define
$$\Gamma(n):=\Gamma(2^n,\overline{2^n}).$$

We set up now some notation that will be useful throughout the proof. 
Let $\sigma_1, \dots , \sigma_l$ be a sequence  of elements of $2^n\cup \overline{2^n}$. 
Define 
\begin{equation}\label{E:mult}
m(\sigma_1,\dots, \sigma_l) := (m_s)_{s\in 2^n} \;\hbox{ and }\; d(\sigma_1,\dots, \sigma_l) := (p,q), 
\end{equation}
where 
\[
m_s :=|\{ i\colon \sigma_i=s \hbox{ or } \sigma_i= \overline{s}\}|, \; 
p := |\{ i\colon \sigma_i\in 2^n\}|,\hbox{ and }\, q :=  |\{ i\colon \sigma_i\in \overline{2^n}\}|.
\]
Note that 
\begin{equation}\label{E:suml} 
\sum_{s\in 2^n} m_s= p+q=l.
\end{equation}
We say a pair of sequences $\big((r_1, \dots, r_p), (t_1, \dots , t_q)\big)$ of elements of $2^n$ is a {\bf variant of} 
$(\sigma_1, \dots, \sigma_l)$ if $(r_1, \dots, r_p)$ is a permutation of the sequence $(\sigma_i\colon \sigma_i\in 2^n)$ and 
$(t_1, \dots, t_q)$ is a permutation of the sequence $(\overline{\sigma_i}\colon \sigma_i\in \overline{2^n})$.
For ease of computation, it will be useful to have the following convention for appending a binary digit to
an element of $\overline{2^n}$.
If $\sigma = \overline{s}$ 
for some $s \in 2^n$ and $\epsilon \in \{0,1\}$, then set
\[
\sigma\epsilon := \overline{s\epsilon}. 
\]

The following lemma contains the relevant to us facts on the norm in the $\Gamma(n)$. 

\begin{lemma}\label{L:norep} 
Let $v_{\sigma_1}\cdots v_{\sigma_l} \in \Gamma(n)$ for some $n,l\in \Nbb$ and 
let $m( \sigma_1\cdots \sigma_l) = (m_s)_{s\in 2^n}$. 
The following are true:
\begin{enumerate}[(i)]
\item \label{norm_comp} $\| v_{\sigma_1}\cdots v_{\sigma_l}\|^2 = \prod_{s\in 2^n} (m_s!)$

\item \label{cat_norm_comp}
For $\epsilon \in 2^l$ 
\[
\|  v_{\sigma_1\epsilon_1}\cdots v_{\sigma_l\epsilon_l}\|^2 = \prod_{s\in 2^n} (k_s!(m_s-k_s)!), 
\]
where 
$k_s$ is the number of $i$ such that $\epsilon_i =0$ and either $\sigma_i=s$  or $\sigma_i=\overline{s}$. 

\item \label{norm_cr} 
${\displaystyle 2^l \| v_{\sigma_1}\cdots v_{\sigma_l}\|^2 =  \|\sum_{\epsilon\in 2^l} v_{\sigma_1\epsilon_1}\cdots v_{\sigma_l\epsilon_l}\|^2}$
\end{enumerate}
\end{lemma} 

\begin{proof}
(\ref{norm_comp}) is a restatement of \eqref{E:nor}, and 
(\ref{cat_norm_comp}) follows from (\ref{norm_comp}) after noticing that
$m(\sigma_1\epsilon_1, \dots, \sigma_l\epsilon_l) = (m'_t)_{t\in 2^{n+1}}$, where, 
for $s\in 2^n$, $m'_{s0} = k_s$ and $m'_{s1} = m_s-k_s$. 

We now show (iii). For two sequences $\epsilon,\epsilon' \in 2^l$, the two vectors of the form
$v_{\sigma_1\epsilon_1}\cdots v_{\sigma_l\epsilon_l}$ and
$v_{\sigma_1\epsilon_1'}\cdots v_{\sigma_l\epsilon_l'}$
are equal precisely when 
the values $k_s$ computed for one of them as in \eqref{cat_norm_comp} are equal to 
the values computed for the other one; otherwise, the two vectors are orthogonal. Additionally, for 
a given sequence $(k_s)_{s\in 2^n}$ with $0\leq k_s\leq m_s$, there are 
$\prod_{s\in 2^n} \binom{m_s}{k_s}$ such equal to each other vectors. With these observations in mind, we 
compute $\| \sum_{\epsilon\in 2^l}  v_{\sigma_1\epsilon_1}\cdots v_{\sigma_l\epsilon_l}\|^2$. We 
let $s$ range over $2^n$ and $(k_s)_s$ range over all sequences indexed by $s\in 2^n$ such that 
$0\leq k_s\leq m_s$ for all $s$. Using (\ref{cat_norm_comp}) to get the first equality, the identity 
$\sum_s m_s=l$ from \eqref{E:suml} to get the next to last equality, and (\ref{norm_comp}) to get the last one, we have 
\[
\begin{split}
\| \sum_{\epsilon\in 2^l}  v_{\sigma_1\epsilon_1}\cdots v_{\sigma_l\epsilon_l}\|^2 
&= \sum_{(k_s)_s} \left[ \left( \prod_s \binom{m_s}{k_s}\right)^2 \left(\prod_s k_s!(m_s-k_s)!\right)\right]\\
&=  \sum_{(k_s)_s} \left[\prod_s \binom{m_s}{k_s} \prod_s (m_s!)\right]\\ 
&=  \left( \sum_{(k_s)_s} \prod_s \binom{m_s}{k_s} \right) \, \prod_s (m_s!)\\
&= 2^l \prod_s (m_s!) = 2^l \|v_{\sigma_1} \cdots v_{\sigma_l}\|^2.\qedhere
\end{split} 
\]
\end{proof}

\subsection*{Definition of $E_n$}

We define a Hilbert space embedding $E_n\colon \Gamma(n)\to \Gamma(n+1)$. We specify it first on basic product vectors 
by letting 
\[
E_n(v_{\sigma_1}\cdots v_{\sigma_l}) = \bigl(\frac{1}{\sqrt{2}}\bigr)^l  \sum_{\epsilon\in 2^l}  v_{\sigma_1\epsilon_1}\cdots v_{\sigma_l\epsilon_l},
\]
for a basic product vector $v_{\sigma_1}\cdots v_{\sigma_l}$ in $L^2(2^n\cup \overline{2^n})^{\od l}$ with $l\geq 1$. 
Note that $E_n(v_{\sigma_1}\cdots v_{\sigma_l})$ is a vector in $L^2(2^{n+1}\cup\overline{2^{n+1}})^{\od l}$. Clearly, 
$E_n$ is induced by the assignment 
\begin{equation}\notag
\begin{split}
v_s\to \frac{v_{s0}+v_{s1}}{\sqrt{2}} \qquad \qquad \qquad
v_{\overline s}\to \frac{v_{\overline{s0}}+v_{\overline{s1}}}{\sqrt{2}}
\end{split}
\end{equation}
from $L^2(2^n\cup \overline{2^n})$ to $L^2(2^{n+1}\cup \overline{2^{n+1}})$.

\begin{lemma}\label{L:En_emb}
$E_n$ extends to a Hilbert space embedding $\Gamma(n)\to \Gamma(n+1)$, which we again denote by $E_n$. 
\end{lemma}

\begin{proof} 
It suffices to show that $E_n$ preserves the norm and orthogonality of basic product vectors. 
Orthogonality of images of basic product vectors is immediate from the observation that distinct 
basic product vectors are orthogonal. 
The preservation of norm follows from Lemma~\ref{L:norep}\eqref{norm_cr}.
\end{proof}

\subsection*{Definition of $\psi_n$} 

Let $g= (g_s)_{s\in 2^n}\in  \Sbb_n$ and let $v_{\sigma_1} \cdots v_{\sigma_l}$ be a basic product
in $\Gamma(n)$. 
Let $\psi_n(g):\Gamma(n) \to \Gamma(n)$ be the linear map determined by 
\[
\psi(g)(v_{\sigma_1} \cdots v_{\sigma_l}) := 
\big(g_{r_1} \cdots g_{r_p}\cdot g^{-1}_{t_1} \cdots g^{-1}_{t_q}\big) \, v_{\sigma_1} \cdots v_{\sigma_l}, 
\]
where $\big((r_1, \dots, r_p), (t_1, \dots , t_q)\big)$ is some variant of $(\sigma_1, \dots, \sigma_l)$.
Note that since $\Tbb$ is commutative, this definition does not depend on the choice of the variant. 
Clearly $\psi_n$ is a unitary representation of $\Sbb_n$ on $\Gamma(n)$ and 
$E_n\colon \Gamma(n)\to \Gamma(n+1)$ is $\Sbb_n$-equivariant between $\psi_n$ 
and $\psi_{n+1}\res \Sbb_n$. 
This completes Step \ref{GammaEpsi} in our proof of Theorem~\ref{T:exp}.

\subsection*{Step \ref{F^a_n}: The maps $F_n^\alpha:\Gamma(n) \to A$ and their properties}\label{S:dirs}

Suppose that $\sigma_1,\dots,\sigma_l$ is an admissible sequence from $2^n \cup \overline{2^n}$.
Set $(p,q):= d(\sigma_1, \dots, \sigma_l)$ and define
\begin{equation}\label{E:setsym}
\supt{\sigma_1, \dots, \sigma_l} \subseteq (2^\Nbb)^p\times (2^\Nbb)^q= (2^\Nbb)^l
\end{equation}
to be the union of all sets of the form 
\begin{equation}\label{E:setch}
[r_1]\times \cdots \times [r_p] \times [t_1]\times \cdots \times [t_q], 
\end{equation} 
where $\big((r_1, \dots, r_p), (t_1, \dots, t_q)\big)$ ranges over all variants of $\sigma_1, \dots, \sigma_l$. 

We define a Hilbert space embedding $F^\alpha_n\colon \Gamma(n)\to A$. If
$v_{\sigma_1}\cdots v_{\sigma_l}$ is a basic product vector in $\Gamma(n)$, then let 
\[
F^\alpha_n(v_{\sigma_1}\cdots v_{\sigma_l}) = 
\sqrt{\frac{2^{nl}}{p!\,q!}} \prod_{s\in 2^n} (m_s!) \cdot \chi_{\supt{\sigma_1,\dots,\sigma_l}},
\]
where $(m_s)_{s\in 2^n}= m(\sigma_1\cdots \sigma_l)$ and
$(p,q) = d(\sigma_1\cdots \sigma_l)$.
Note that $F^\alpha_n(v_{\sigma_1}\cdots v_{\sigma_l})$ is an element of $\widetilde{L^2}(\lambda^p\times \lambda^q)$. 

\begin{lemma}\label{P:exta} 
$F_n^\alpha$ extends to a Hilbert space embedding $\Gamma(n) \to A$, which we again denote by $F^\alpha_n$.
\end{lemma}

\begin{proof}
It is sufficient to show that $F_n^\alpha$ preserves the norm and orthogonality of
the basic product vectors.
In order to see that $F_n^\alpha$ preserves the norm of basic product vectors,
we will first show that 
\begin{equation}\label{E:meas}
\lambda^l\big( \supt{ \sigma_1, \dots, \sigma_l}\big) = \frac{p!\, q!}{2^{nl}\prod_{s\in 2^n} (m_s!)}.
\end{equation} 
Let $\big((r_1, \dots, r_p), (t_1, \dots, t_q)\big)$ be a variant of $(\sigma_1, \dots, \sigma_l)$. 
For function $f$ and $g$ that permute sequences $(1, \dots, p)$ and $(1, \dots, q)$, respectively, we consider the set 
\begin{equation}\label{E:frt}
[r_{f(1)}]\times\cdots \times [r_{f(p)}]\times [t_{g(1)}]\times \cdots \times [t_{g(q)}]. 
\end{equation} 
Each such set has $\lambda^l$-measure $2^{-nl}$, such sets are either disjoint or coincide, and their union over all 
$f, g$ is equal to $\supt{ \sigma_1, \dots, \sigma_l}$. Thus, to show \eqref{E:meas}, it suffices to see that 
there are $p!\, q!/\prod_{s\in 2^n} (m_s!)$ distinct sets of this form. This is true since pairs of permutations $(f, g)$ as above form a group 
(under coordinate-wise composition) of size $p!\, q!$, while the elements of this group for which the set \eqref{E:frt} is equal to 
\[
[r_1]\times\cdots \times [r_p]\times [t_1]\times \cdots \times [t_q] 
\]
form a subgroup of size $\prod_{s\in 2^n} (m_s!)$.
Using \eqref{E:meas}, we compute 
\begin{align*}
\|F^\alpha_n(v_{\sigma_1} \cdots v_{\sigma_l}) \|^2 & =
\frac{2^{nl}}{p!\,q!} \left( \prod_{s\in 2^n} (m_s!) \right)^2 \cdot \lambda^l\big( \supt{ \sigma_1, \dots, \sigma_l}\big)  \\
& = \frac{2^{nl}}{p!\,q!} \left( \prod_{s\in 2^n} (m_s!) \right)^2 \cdot
\frac{p!\, q!}{2^{nl}\prod_{s\in 2^n} (m_s!)} \\
& = \prod_{s\in 2^n} (m_s!) 
\end{align*}
showing that norm is preserved by $F^\alpha_n$. 

To see that $F_n^\alpha$ preserves the orthogonality of basic product vectors, suppose that $v_{\sigma_1}\cdots v_{\sigma_l}$ and $v_{\sigma'_1}\cdots v_{\sigma'_{l'}}$ are basic product
vectors in $\Gamma(n)$.
There are two cases. 
If $d(\sigma_1, \dots, \sigma_l) \not= d(\sigma'_1, \dots, \sigma'_{l'})$, then 
$F_n^\alpha(v_{\sigma_1}\cdots v_{\sigma_l})$ and $F_n^\alpha(v_{\sigma'_1}\cdots v_{\sigma'_{l'}})$ lie in distinct summands 
of the form $\widetilde{L^2}(\lambda^p\times \lambda^q)$, making them orthogonal. 
If $d(\sigma_1, \dots, \sigma_l) = d(\sigma'_1, \dots, \sigma'_{l'})=(p,q)$, then both 
$F_n^\alpha(v_{\sigma_1}\cdots v_{\sigma_l})$ and $F_n^\alpha(v_{\sigma'_1}\cdots v_{\sigma'_{l'}})$ lie in 
$\widetilde{L^2}(\lambda^p\times \lambda^q)$. Their supports 
$\supt{ \sigma_1, \dots, \sigma_l}$ and $\supt{ \sigma'_1, \dots, \sigma'_{l'}}$ are disjoint as shown by 
a quick analysis of the sets \eqref{E:setch} for $v_{\sigma_1}\cdots v_{\sigma_l}$ and $v_{\sigma'_1}\cdots v_{\sigma'_{l'}}$. 
\end{proof}

\begin{lemma}
For each $n$, 
$F^\alpha_n\colon \Gamma(n)\to A$ is equivariant between $\psi_n$ and $\alpha \res \Sbb_n$. 
\end{lemma}

\begin{proof}
It suffices to check equivariance at elements of our basis.
Suppose that $g$ is in $\Sbb_n$ and $v_{\sigma_1}\cdots v_{\sigma_l}$ is a basic product vector in $\Gamma(n)$.
Let $d$ be such that $F^\alpha_n(v_{\sigma_1} \cdots v_{\sigma_l}) = d\cdot \chi_{\supt{\sigma_1,\dots,\sigma_l}}$, 
and let $c$ be the common value of 
\[
g_{r_1} \cdots g_{r_p} \cdot g_{t_1}^{-1} \cdots g_{t_q}^{-1}
\]
for variants 
$\big((r_1,\dots,r_p),(t_1,\dots,t_q)\big)$ of $(\sigma_1,\dots,\sigma_l)$. 

If $(a_1, \dots, a_p, b_1, \dots, b_q)$ is an element of $\supt{\sigma_1,\dots,\sigma_l}$, then it belongs to the set 
\[
[r_1]\times \cdots\times [r_p]\times [t_1]\times \cdots \times[t_q]
\]
for some variant $\big((r_1,\dots,r_p),(t_1,\dots,t_q)\big)$ of $(\sigma_1,\dots,\sigma_l)$. Therefore, 
\[
g(a_1)\cdots g(a_p)\cdot g(b_1)^{-1}\cdots g(b_q)^{-1} = g_{r_1} \cdots g_{r_p} \cdot g_{t_1}^{-1} \cdots g_{t_q}^{-1} = c, 
\]
and this value does not depend on $(a_1, \dots, a_p, b_1, \dots, b_q)$ in $\supt{\sigma_1,\dots,\sigma_l}$. 

It now follows from (\ref{rho_eq}) that 
\[
\alpha(g)(\chi_{\supt{\sigma_1,\dots,\sigma_l}}) = \rho_{p,q}(\chi_{\supt{\sigma_1,\dots,\sigma_l}}) = c \cdot 
\chi_{\supt{\sigma_1,\dots,\sigma_l}},
\]
from which we get 
\begin{equation}\label{E:eqap}
\alpha(g)(F^\alpha_n(v_{\sigma_1},\dots,v_{\sigma_l})) = (c \cdot d) \cdot 
\chi_{\supt{\sigma_1,\dots,\sigma_l}}.
\end{equation}
On the other hand, by definition of $\psi_n$ and linearity of $F^\alpha_n$, we have
\[
F^\alpha_n(\psi_n(g)(v_{\sigma_1} \cdots v_{\sigma_l})) =  F^\alpha_n(c \cdot
v_{\sigma_1} \cdots v_{\sigma_l})
= (c \cdot d) \cdot 
\chi_{\supt{\sigma_1,\dots,\sigma_l}}, 
\]
which together with \eqref{E:eqap} imply equivariance. 
\end{proof}

\begin{lemma} 
$F_n^\alpha = F_{n+1}^\alpha \circ E_n$ 
\end{lemma}

\begin{proof}
We need to see that 
\begin{equation}\label{E:nefr}
F_n^\alpha(v_{\sigma_1}\cdots v_{\sigma_l}) = \bigl(\frac{1}{\sqrt{2}}\bigr)^l \sum_{\epsilon\in 2^l}
F_{n+1}^\alpha( v_{\sigma_1\epsilon_1}\cdots v_{\sigma_l\epsilon_l}). 
\end{equation} 
We fix $v_{\sigma_1}\cdots v_{\sigma_l}$ as in \eqref{E:nefr}. 
Recall that the support of $F_n^\alpha(v_{\sigma_1}\cdots v_{\sigma_l})$ is 
\begin{equation}\label{E:lar}
\supt{ \sigma_1, \dots, \sigma_l}, 
\end{equation} 
and the support of $F_{n+1}^\alpha( v_{\sigma_1\epsilon_1}\cdots v_{\sigma_l\epsilon_l})$ for each $\epsilon\in 2^l$ 
is included in \eqref{E:lar}. 
Now consider sets of the form 
\begin{equation}\label{E:prop}
\supt{ \sigma_1\epsilon'_1, \dots, \sigma_l\epsilon'_l}
\end{equation} 
where $\epsilon'\in 2^l$. These sets form a partition of the set in \eqref{E:lar}. Thus, to check \eqref{E:nefr}, 
we need to see that both sides of \eqref{E:nefr} are equal on each set of the form \eqref{E:prop}. 

Fix $\epsilon'\in 2^l$ for the remainder of the proof. Set also 
\[
m(\sigma_1, \dots , \sigma_l) = (m_s)_{s\in 2^n}\;\hbox{ and }\; d(\sigma_1, \dots , \sigma_l) = (p,q). 
\]

The function $F_n^\alpha(v_{\sigma_1}\cdots v_{\sigma_l})$ is constant on \eqref{E:lar}, and therefore on \eqref{E:prop}. We denoted
this constant value by $c$.
For every $\epsilon \in 2^l$, 
the function $F_{n+1}^\alpha( v_{\sigma_1\epsilon_1}\cdots v_{\sigma_l\epsilon_l})$ is also constant on \eqref{E:prop}. 
This constant value is equal to $0$ if 
\[
v_{\sigma_1\epsilon_1}\cdots v_{\sigma_l\epsilon_l} \not= v_{\sigma_1\epsilon'_1}\cdots v_{\sigma_l\epsilon'_l}, 
\]
and is equal to some $d$, which is independent of $\epsilon$, if 
\begin{equation}\label{E:eqcr} 
v_{\sigma_1\epsilon_1}\cdots v_{\sigma_l\epsilon_l} = v_{\sigma_1\epsilon'_1}\cdots v_{\sigma_l\epsilon'_l}.
\end{equation} 
Note that if \eqref{E:eqcr} holds, then the numbers $k_s$, $s\in 2^n$, computed as in Lemma~\ref{L:norep}(ii) 
for $\epsilon \in 2^l$ do not depend on $\epsilon$. Furthermore, 
there are $\prod_{s\in 2^n} \binom{m_s}{k_s}$ 
choices for the sequence $\epsilon$ for which \eqref{E:eqcr} holds and 
$m(\sigma_1\epsilon_1, \dots, \sigma_l\epsilon_l) = (m'_t)_{t\in 2^{n+1}}$, where 
\begin{equation}\label{E:exm}
m'_{s0} = k_s\hbox{ and }m'_{s1} = m_s-k_s, \hbox{ for } s\in 2^n,
\end{equation}
for each such sequence $\epsilon$.

It follows from the information collected above that to check \eqref{E:nefr}, we need 
\begin{equation}\label{E:AA}
c = \left( \frac{1}{\sqrt{2}}\right)^l \left( \prod_{s\in 2^n} \binom{m_s}{k_s} \right) d. 
\end{equation}
By the definition of $F^\alpha_n$ and \eqref{E:exm}, we have 
\[
c= \sqrt{\frac{2^{nl}}{p!\,q!}} \prod_{s\in 2^n} (m_s!) 
 \qquad \hbox{ and } \qquad  
d =  \sqrt{\frac{2^{(n+1)l}}{p!\,q!}} \prod_{s\in 2^n} (k_s!\, (m_s-k_s)!), 
\]
which imply that \eqref{E:AA} becomes
\[
\sqrt{2^{nl}} \prod_{s\in 2^n} (m_s!) =
\left( \frac{1}{\sqrt{2}}\right)^l \left( \prod_{s\in 2^n} \binom{m_s}{k_s} \right)\sqrt{2^{(n+1)l}} \prod_{s\in 2^n} (k_s!\, (m_s-k_s)!).
\]
Checking this equality boils down to the tautology 
\begin{equation*}
m_s! =\binom{m_s}{k_s} (k_s!\, (m_s-k_s)!). \qedhere
\end{equation*}
\end{proof}

The next lemma completes our analysis of the maps $F^\alpha_n$.

\begin{lemma}
$\bigcup_n F_n^\alpha\big(\Gamma(n)\big)$ is dense in $A$.
\end{lemma}

\begin{proof} Fix $p, q$ with $p, q\geq 0$ and $p+q>0$. Let $\leq_{\rm lx}$ be the standard 
lexicographic order on $2^\Nbb$. It suffices to show that functions in 
$\bigcup_n F_n^\alpha\big(\Gamma(n)\big)$ separate points of the set $K$ of all $(a_1, \dots, a_p, b_1, \dots, b_q)$ in 
$(2^\Nbb)^p\times (2^\Nbb)^q$ such that 
\[
a_1\leq_{\rm lx} \cdots \leq_{\rm lx} a_p,\, b_1\leq_{\rm lx} \cdots \leq_{\rm lx} b_q, \hbox{ and } a_i\not= b_j,\hbox{ for all } 
1\leq i\leq p,\, 1\leq j\leq q
\]
since the complement of $K$ in the set of all 
$(a_1, \dots, a_p, b_1, \dots, b_q)$ in $(2^\Nbb)^p\times (2^\Nbb)^q$ with 
\[
a_1\leq_{\rm lx} \cdots \leq_{\rm lx} a_p\hbox{ and } b_1\leq_{\rm lx} \cdots \leq_{\rm lx} b_q
\]
has $\lambda^{p+q}$-measure zero. The separation of points is evident since sets of the form 
\[
\supt{ \sigma_1, \dots, \sigma_{p+q}} \cap K,
\]
where 
\[
\sigma_i = a_i\res n, \hbox{ for }1\leq i\leq p,\; \hbox{ and }\; \sigma_{p+j}= \overline{b_j\res n}, \hbox{ for } 1\leq j\leq q,
\]
with $n$ large enough so that $a_i\res n \not= b_j\res n$, for all $1\leq i\leq p,\, 1\leq j\leq q$,
constitute a topological basis in $K$ at the point $(a_1, \dots, a_p, b_1, \dots, b_q)\in K$. 
\end{proof}

\subsection*{Step \ref{F^b_n}: The maps $F_n^\beta:\Gamma(n) \to B$ and their properties}\label{S:boo} 

If $s \in 2^n$ for some $n$, define $z_{\bar s} = \bar z_s$.
We define a Hilbert space embedding $F_n^\beta:\Gamma(n) \to B$. 
For each basic product vector $v_{\sigma_1}\cdots v_{\sigma_l}$ in $\Gamma(n)$, set 
\[
F^\beta_n(v_{\sigma_1}\cdots v_{\sigma_l}) : =  \prod_{i=1}^l z_{\sigma_i}.
\] 
\begin{lemma}\label{L:extb} 
$F^\beta_n$ extends to a Hilbert space embedding $\Gamma(n) \to B$, which we again denote by $F^\beta_n$. 
\end{lemma}

\begin{proof}
It suffices to show that $F^\beta_n$ preserves the norm
and orthogonality of the basic vectors $v_{\sigma_1}\cdots v_{\sigma_l}$ from $\Gamma(n)$.
By
\eqref{inner_prod} and the independence of the random variables
$\{z_s \mid s \in 2^n\}$, 
we have that if $m(\sigma_1,\cdots ,\sigma_l) = (m_s)_{s\in 2^n}$, then 
\[
 \| F^\beta_n( v_{\sigma_1}\cdots v_{\sigma_l}) \|^2 = 
\int_{X_\infty} |z_{\sigma_1}\cdots z_{\sigma_l}|^2\, d\gamma_\infty 
= \prod_{s\in 2^n} \int_{X_\infty} |z_s|^{2m_s} \,d\gamma_\infty
= \prod_{s\in 2^n}  (m_s!).
\]
So $F^\beta_n$ preserves
the norm of basic product vectors in $\Gamma(n)$.
Furthermore, \eqref{inner_prod} implies any two distinct vectors of the form
$z_{\sigma_1} \cdots  z_{\sigma_l}$ with $l \geq 1$ are orthogonal. 
\end{proof}

It is immediate from the definitions that, for each $n \in \Nbb$,
$F^\beta_n\colon \Gamma(n)\to B$ is $\Sbb_n$-equivariant
between $\psi_n$ and $b\res \Sbb_n$ and also that
$F_n^\beta = F_{n+1}^\beta\circ E_n$.
The next lemma completes the proof of Theorem~\ref{T:exp}.

\begin{lemma}\label{L:offde} 
$\bigcup_n F^\beta_n(\Gamma(n))$ is dense in $B$.
\end{lemma}

\begin{proof}
Let $C$ denote the collection of all constant functions in $L^2(\gamma_\infty)$,
noting that $L^2(\gamma_\infty) = B+C$. 
If $S \subseteq 2^n$, let $P_S$ denote the span of all products
of the form $\prod_{s} z_s^{m_s}$
where $s$ ranges over $S \cup \overline{S}$ and
$m_s \geq 0$ for all such $s$.
Here we allow the product of length 0, which we define to be equal to 1.
Define $Q_S \subseteq P_S$ to consist of the span all such products which moreover satisfy that
$m_s + m_{\bar s} \leq 1$ for all $s \in S$.
Both $P_S$ and $Q_S$ contain $C$ by convention.
If $S = 2^n$, we will write $P$ and $Q$ for $P_S$ and $Q_S$.

In order to show that $\bigcup_n F^\beta_n(\Gamma(n))$ is dense in $B$, it suffices to show that
$\bigcup_n F^\beta_n(\Gamma(n)) + C$ is dense in $B + C = L^2(\gamma_\infty)$.
Since $Q \subseteq \bigcup_n F^\beta_n(\Gamma(n)) + C$, it will be sufficient to show that
$Q$ is dense in $P$ and $P$ is dense in $L^2(\gamma_\infty)$.

\begin{claim}\label{separated_product}
Suppose that $S_0,S_1 \subseteq 2^n$ are disjoint sets.
If $f_i$ is in $P_{S_i} \cap \cl(Q_{S_i})$, 
then $f_0 f_1$ is in $\cl(Q_{S_0 \cup S_1})$.
\end{claim}

\begin{proof}
Let $1 \geq \epsilon > 0$ be arbitrary.
For $i=0,1$ let $g_i \in Q_{S_i}$ be such that $\|f_i - g_i\| < \epsilon/(\|f_0\| + \|f_1\| + 1)$.
Note in particular that since $\|g_1\| \leq \|f_1 - g_1\| + \|f_1\|$, this implies $\|g_1\| \leq \|f_1\| + 1$.
Since $f_0$ and $f_1-g_1$ are independent, we have that
\begin{equation}\notag
\|f_0 f_1 - f_0 g_1\| = \|f_0\| \cdot \|f_1 - g_1\| < \|f_0\| \frac{\epsilon}{\|f_0\| + \|f_1\| + 1}.
\end{equation} 
Similarly, since $g_1$ and $f_0 - g_0$ are independent
\begin{equation}\notag
\|f_0 g_1 - g_0 g_1\| = \|f_0- g_0\| \cdot \|g_1\| < \frac{\epsilon}{\|f_0\| + \|f_1\| + 1} (\|f_1\| + 1) .
\end{equation} 
Thus, 
\[
\|f_0 f_1 - g_0 g_1\|  <
\frac{\|f_0\| \epsilon}{\|f_0\| + \|f_1\| + 1} + \frac{(\|f_1\|+1) \epsilon}{\|f_0\| + \|f_1\| + 1} = \epsilon.
\]
Since $\epsilon$ was arbitrary, $f_0 f_1$ is in the closure of $Q_{S_0 \cup S_1}$. 
\end{proof}

\begin{claim}\label{Q_dense_base}
For any $s \in 2^n$, $z_s^k  \bar z_s^m$ is in $\cl(Q_{\{s\}})$.
\end{claim}

\begin{proof}
The proof is by induction on $k+m$.
The base cases $k+m = 0,1$ follow from the fact that $z_s$, $\bar z_s$, and the constant functions are in $Q_S$.
Suppose that $k+m > 1$ and let $l \geq 0$ be fixed for the moment.
Observe 
$$
z_s^k \bar z_s^m = 2^{-l(k+m)/2} \sum_{\vec{t} \in (2^l)^{k+m}} \prod_{i =1}^k z_{s t_i} \prod_{i=1}^m \bar z_{s t_{k+i}}
$$
If $\vec{t} \in (2^l)^{k+m}$ is not a constant sequence, then let $\Seq{u_i \mid 1 \leq i \leq j}$ be an enumeration
of $\{t_i \mid 1 \leq i \leq k+m\}$ without repetition so that
$$
\prod_{i=1}^k z_{s t_i} \prod_{i=1}^m \bar z_{s t_{k+i}}
=
\prod_{i=1}^j z_{s u_i}^{m_i} \bar z_{s u_i}^{n_i}
$$
where $0 \leq k_i , m_i$ and $0 < k_i + m_i$ for each $1 \leq i \leq j$.
Note that $j > 1$ by our assumption that $\vec{t}$ is not constant and since $k+m = \sum_{i=1}^j k_i + m_i$ with each term positive,
it must be that $k_i + m_i < k + m$ for all $1 \leq i \leq j$.
Thus by our induction hypothesis, $z_{s u_i}^{k_i} \bar z_{s u_i}^{m_i}$ is in $Q_{\{s u_i\}}$ for each
$1 \leq i \leq j$.
By iteratively applying Claim \ref{separated_product},
$$
\prod_{i=1}^k z_{s t_i} \prod_{i=1}^m \bar z_{s t_{m+i}}
=
\prod_{i=1}^j z_{s u_i}^{k_i} \bar z_{s u_i}^{m_i}
$$
is in $\cl(Q_{\{s u_i \mid 1 \leq i \leq j\}}) \subseteq \cl(Q_{\{s\}})$.

The sum of the remaining terms from the expansion of $z_s^k \bar z_s^m$ is 
$$r_l = 2^{-l(k+m)/2} \sum_{t \in 2^l} z_{s t}^k \bar z_{st}^m.$$
Since $Q_S$ contains the constant functions $C$, it suffices to show that
$r_l -c \to 0$ converges to 0 for some $c \in C$.
 
If $k \ne m$, then by \eqref{inner_prod} $\{z_{st}^k \bar z_{st}^m \mid t \in 2^l \}$ 
consists of elements of $B$ and is pairwise orthogonal.
Moreover
$$\|r_l\|^2 = 2^{-l(k+m)} \sum_{t \in 2^l} \|z_{st}^k \bar z_{st}^m\|^2 = (k+m)!\ 2^{-l(k+m-1)}.$$
In this situation as $l \to \infty$, $\|r_l\| \to 0$.

If $k = m$, then by \eqref{inner_prod} $\{z_{st}^m \bar z_{st}^m - \sqrt{m!} \mid t \in 2^l\}$ consists of elements of $B$ and is pairwise orthogonal.
Moreover $$\|r_l - \sqrt{m!}\|^2 =
2^{-2lm} \sum_{t \in 2^l} \|z_{st}^m \bar z_{st}^m - \sqrt{m!}\|^2 = 2^{-l (2m-1)} ((2m)! - (m!)^2). 
$$
Again, as $l \to \infty$, $\|r_l - \sqrt{m!}\| \to 0$.
\end{proof}

\begin{claim} \label{Q_dense}
$Q$ is dense in $P$.
\end{claim}

\begin{proof}
Since $P$ is closed under linear combinations, it suffices to show that if $S \subseteq 2^k$ for some $k$,
then any product of the form $\prod_{s \in S} z_s^{m_s} \bar z_s^{n_s}$ is in $\cl(Q_S)$.
This is established by induction on the cardinality of $S$.
If $S$ is empty, this is a consequence of the fact that $Q$ contains all constant functions.
If $S$ is non-empty, let $t \in S$ and set $S' = S \setminus \{t\}$.
By induction $\prod_{s \in S'}  z_s^{m_s} \bar z_s^{n_s}$ is in $P_{S'} \cap \cl(Q_{S'})$ and
by Claim \ref{Q_dense_base}, $z_t^{m_t} \bar z_t^{n_t}$ is in $P_{\{t\}} \cap \cl(Q_{\{t\}})$.
By Claim \ref{separated_product}, $\prod_{s \in S} z_s^{m_s} \bar z_s^{n_s}$ is in $\cl(Q_S)$.
\end{proof}

Now we finish the proof of the lemma.
As noted above, it is now sufficient to show that $P$ is dense in $L^2(\gamma_\infty)$.
Since $X_n$ is locally compact, the Stone-Weierstrass theorem implies that the polynomials
in $\{z_s,\bar z_s \mid s \in 2^n\}$ are dense in $L^2(\gamma_n)$.
Since $X_\infty$ is an inverse limit of the spaces $X_n$,
$P$ is dense
in $L^2(\gamma_\infty)$.
\end{proof}

\section{Constraints on the spectral form of Koopman representations}\label{Su:sem} 

In order to state the constraint on the spectral forms of Koopman representations of
$L^0(\lambda, \Tbb)$ proved in \cite{Sol2} 
and explain 
its connection with Theorem~\ref{T:koo}, we need to introduce semigroup structures on the objects involved in spectral forms. 

We can identify $\Nbb[\Zbb^\times]$ with the collection of all finitely supported
functions from $\Zbb^\times$ into $\Nbb \cup \{0\}$. This identification allows us to regard $\Nbb[\Zbb^\times]$ as a semigroup
with the operation $\oplus$ of coordinatewise addition.
There is an action of the semigroup $\Zbb^\times$ on
$\Nbb[\Zbb^\times]$.
For $x\in \Nbb[\Zbb^\times]$ and 
$m \in \Zbb^\times$, let $m x$ be the element $\Nbb[\Zbb^\times]$ such that 
${\rm dom}(m x) = \{ mn \mid n\in {\rm dom}(x)\}$ and,  for $n\in {\rm dom}(m x)$, 
\[
(m x)(n) := x(n/m). 
\]

The space $C_{x\oplus y}$ can be identified with $C_x\times C_y$ in several ways. To make these identifications precise, 
define $I$ to be the set of all ordered pairs
${\bar \iota}= (\iota^x, \iota^y)$ such that
\[
\iota^x\colon D(x)\to D(x\oplus y) \qquad \qquad 
\iota^y\colon D(y)\to D(x\oplus y)
\]
are injections which fix the first coordinate and whose ranges partition $D(x \oplus y)$.
For each $\bar \iota \in I$, define a homeomorphism $h_{\bar\iota}\colon C_x\times C_y \to  C_{x\oplus y}$ by
\[
h_{\bar\iota}(a, b) (k,j):=
\begin{cases}
a(k,i) & \textrm{ if } \iota^x(k,i) = (k,j) \\
b(k,i) & \textrm{ if } \iota^y(k,i) = (k,j)
\end{cases}
\]
For $\mu$ compatible with $x$ and $\nu$ compatible with $y$, we define 
\[
\mu\otimes \nu = 
\sum_{\bar\iota \in I} (h_{\bar\iota})_*(\mu\times \nu).
\]
It was proved in \cite{Sol2} that if $\mu$ and $\nu$ are measures compatible with $x$ and $y$, respectively, then 
$\mu\otimes \nu$ is compatible with $x\oplus y$.

There is a natural homeomorphism between $C_x$ and $C_{mx}$. 
For $m \in \Zbb^\times$, define $e_{x,m}\colon C_x\to  C_{m x}$ by
$$e_{x,m}(a)(m k,i) := a(k,i).$$ 
For $\mu$ compatible with $x$, set 
$m \mu := (e_{x,m})_*(\mu)$.
It was observed in \cite{Sol2} that if $\mu$ is a measure compatible with $x$ and $m \in \Zbb^\times$, 
then $m\mu$ is compatible with $m x$.

Theorem~\ref{T:last}, proved in \cite{Sol2}, gives the additional condition fulfilled by the spectral form in case 
the unitary representation of $L^0(\lambda, \Tbb)$ is a Koopman representation.

\begin{theorem}[\cite{Sol2}]\label{T:last}
Assume that $(\mu^j_x)_{x,j}$
determines the spectral form of 
a Koopman representation associated with  an ergodic boolean action of $L^0(\lambda, \Tbb)$. 
Then, for $m_1, \dots , m_n\in \{1,-1\}$
and $x_1, \dots , x_n\in \Nbb[\Zbb^\times]$, 
we have 
\begin{equation}\label{E:cond}
m_1\mu^1_{x_1} \otimes \cdots \otimes m_n\mu^1_{x_n} \preceq \mu^1_x, \hbox{ where }x= m_1x_1 \oplus \cdots \oplus  m_n x_n.
\end{equation}
\end{theorem}

In Theorem~\ref{T:last} the coefficients $m_1, \dots , m_n$ come from $\{1,-1\}$.
It is natural to enquire, especially in light of the arguments in \cite{Sol2} and the main theorem of \cite{Ete} 
as reformulated in \cite[Section~4]{Sol2}, if the conclusion 
can be strengthened so that \eqref{E:cond} holds with arbitrary coefficients from $\Zbb^\times$.  
Theorem~\ref{T:koo} implies that this is impossible. 

Indeed, let $(\mu^j_x)_{x,j}$
determine the spectral form of the Koopman 
representation associated with the boolean action of $L^0(\lambda,\Tbb)$ from Section~\ref{Su:act}.
It follows from Theorem~\ref{T:koo} 
that for $m_1, \dots, m_n \in \Zbb^\times$ and $x_1, \dots, x_n\in \Nbb[\Zbb^\times]$, 
condition \eqref{E:cond} holds if and only if either each $m_i = \pm1$ or
else there exists $1\leq i\leq n$ such that $m_i \ne \pm1$ and 
$x_i\not= x_{p,q}$, for all $p,q$.
In the latter case $m_1 \mu^1_{x_1}\otimes \cdots \otimes m_n\mu^1_{x_n}$ is the zero measure.
In particular, for $m\in \Zbb^\times$ we have that 
\[
\big( m \mu^1_{x_{1,0}}\preceq \mu^1_x,\hbox{ for }x= {m x_{1,0}}\big) \,\Longleftrightarrow\,   \big(m \in \{-1, 1\}\big).
\]

%
%
%

\end{document}